\documentclass[12pt]{amsart}

\usepackage{verbatim, amssymb, enumerate,mathdots}

\usepackage[breaklinks=true,colorlinks=true]{hyperref}

\allowdisplaybreaks

\renewcommand \a{\alpha}
\renewcommand \b{\beta}

\newcommand \la{\lambda}
\newcommand \ve{\varepsilon}

\newcommand \br{\mathbb{R}}
\newcommand \bc{\mathbb{C}}
\newcommand \M{\mathbb{M}}
\newcommand \E{\mathbb{E}}
\newcommand \Rn{\mathbb R^n}

\newcommand \Ric{\operatorname{Ric}}

\newcommand \jt{\mathrm{j}}
\newcommand \Lc{\mathcal{L}}

\newcommand{\Hess}{\operatorname{Hess}}

\theoremstyle{plain}
\newtheorem{theorem}{Theorem}
\newtheorem*{theorem*}{Theorem}
\newtheorem*{corollary*}{Corollary}
\newtheorem*{conj*}{Conjecture}
\newtheorem{lemma}{Lemma}
\newtheorem*{lemma*}{Lemma}
\newtheorem{proposition}{Proposition}
\newtheorem*{prop*}{Proposition}

\theoremstyle{definition}

\newtheorem*{definition*}{Definition}

\theoremstyle{remark}

\begin{document}

\title[Cylindricity of submanifolds in Minkowski space]{On cylindricity of submanifolds of nonnegative Ricci curvature in a Minkowski space}

\author{A.~Borisenko}
\address{Department of Differential Equations and Geometry, B.Verkin Institute for Low Temperature Physics and Engineering
of the National Academy of Sciences of Ukraine, Kharkiv, 61103, Ukraine}
\email{aborisenk@gmail.com}

\author{Y.~Nikolayevsky}
\address{Department of Mathematics and Statistics, La Trobe University, Melbourne, 3086, Australia}
\email{Y.Nikolayevsky@latrobe.edu.au}




\thanks{The authors were supported by ARC Discovery Grant DP130103485 and by La Trobe University DRP}

\begin{abstract} 
We consider Finsler submanifolds $M^n$ of nonnegative Ricci curvature in a Minkowski space $\M^{n+p}$ which contain a line or whose relative nullity index is positive. For hypersurfaces, submanifolds of codimension two or of dimension two, we prove that the submanifold is a cylinder, under a certain condition on the inertia of the pencil of the second fundamental forms. We give an example of a surface of positive flag curvature in a three-dimensional Minkowski space which is not locally convex.
\end{abstract}

\maketitle

\section{Introduction}
\label{s:intro}

By Cohn-Vossen \cite{CV}, a complete two-dimensional manifold of nonnegative Gauss curvature which contains \emph{a line} (a complete geodesic every arc of which minimises the distance between its endpoints) is flat. Toponogov \cite{Top} generalised this result to higher dimension; he proved that if a complete Riemannian manifold $M^n$ of nonnegative sectional curvature admits $k$ independent lines, then $M^n$ is the Riemannian product $M^{n-k} \times \E^k$, where $\E^k$ is the Euclidean space of dimension $k$. By the Splitting Theorem of Cheeger-Gromoll, the same is true in the case of nonnegative Ricci curvature \cite{CG}. In pseudo-Riemannian settings, the Splitting Theorem was established by Eschenburg \cite{Esch} and Galloway \cite{Gal}. Recently Ohta \cite{Ohta} established the differentiable (resp. the isometric) Splitting Theorem for Finsler (resp. Berwald) manifolds of nonnegative \emph{weighted} Ricci curvature.

In submanifold geometry, one has the following result proved by the first author. Suppose $M^n$ is a complete, regular submanifold of a Euclidean space. If the index of relative nullity at every point of $M^n$ is bounded below by $k \ge 1$ and if the Ricci curvature of $M^n$ is nonnegative, then $M^n$ is a cylinder with a $k$-dimensional generatrix \cite[Theorem~3.3.2]{BorUMN}; a similar result assuming nonnegativity of the sectional curvature has been earlier established by Hartman \cite{Har}. In the Finsler settings, similar results were obtained by the first author for hypersurfaces of a Randers space. He proved that a complete hypersurface of nonnegative Ricci curvature of a Randers space which contains a line is a cylinder \cite[Theorem~3]{BorMS}. If the latter assumption is replaced by the condition that the index of relative nullity at every point of $M^n$ is greater than or equal to $k \ge 1$, then $M^n$ is a cylinder with a $k$-dimensional generatrix \cite[Theorem~4]{BorMS}.

In this paper, we study complete Finsler submanifolds of nonnegative Ricci curvature in an arbitrary Minkowski space. It should be noted that the connection between the sign of the Ricci (or of the flag) curvature of a submanifold in a Minkowski space and the shape of the submanifold is much weaker and is much less understood compared to that for a submanifold of a Euclidean space (see e.g. \cite{BI}).

\bigskip

A \emph{Minkowski space} $\M^n$ is a pair $(\Rn, F)$, where $F:\Rn \to [0, \infty)$ is a continuous function (\emph{Minkowski norm}) satisfying the conditions
\begin{enumerate}
  \item $F \in C^\infty(\Rn \setminus \{0\})$;
  \item $F$ is positive homogeneous, that is, $F(\lambda y) = \lambda F(y)$, for all $\lambda > 0$ and $y \in \Rn$;
  \item The matrix
  \begin{equation*}
    g_{ij}=\frac12 \frac{\partial^2F^2}{\partial y^i \partial y^j},
  \end{equation*}
  where $y^i$ are Cartesian coordinates on $\Rn$, is positive definite outside of the origin of $\Rn$. 
\end{enumerate}

A smooth \emph{Finsler metric} on a smooth manifold is obtained by assigning a Minkowski metric smoothly depending on a point to every tangent space. Given a regular submanifold $M^n$ in a Minkowski space $\M^{n+p}=(\br^{n+p},F)$, the Minkowski norm $F$ induces a Finsler metric on $M^n$. We say that $M^n$ is \emph{complete} if it is (forward and backward) complete for the induced Finsler metric. Clearly, completeness does not depend on the choice of the Minkowski norm $F$ for $\br^{n+p}$ (in particular, one can choose a Euclidean norm).

A submanifold $M^n$ in a Minkowski space $\M^{n+p}$ is called \emph{a cylinder with a $k$-dimensional generatrix, $k \ge 1$}, if $M^n$ is the union of $k$-dimensional affine subspaces of $\M^{n+p}$ parallel to a fixed subspace $\br^k \subset \br^{n+p}$.

Given a point $P$ on a regular submanifold $M^n \subset \M^{n+p}$ we introduce two integer invariants, the index of relative nullity $\mu(P)$ and the type $\jt(P)$. Choose an arbitrary Euclidean metric on $\br^{n+p}$ and define the \emph{null-space} $\Lc(P) \subset T_PM^n$ to be the common kernel of all the shape operators of $M^n$ at $P$. The \emph{index of relative nullity} is defined by $\mu(P)=\dim \Lc(P)$ \cite{CK}. The \emph{type} $\jt(P)$ of the point $P$ is defined to be the minimum of the positive index of inertia of the second fundamental forms at $P$ relative to the normals for which the rank of the shape operator is maximal \cite{BorUGS}. Note that if $\jt(P) = 0$, then there exists a normal for which the second fundamental form at $P$ is positive semidefinite. One can easily verify that $\mu(P), \jt(P)$ and $\Lc(P)$ are affine invariants: they do not depend on the choice of the Euclidean metric on $\br^{n+p}$.


We prove the following.

\begin{theorem} \label{th:hyper}
Let $M^n$ be a complete, connected, smooth, regular hypersurface in a Minkowski space $\M^{n+1}$. Suppose that
\begin{enumerate}[{\rm (a)}]
  \item \label{it:hyperric}
  the Ricci curvature of the induced Finsler metric on $M^n$ is nonnegative;
  \item \label{it:hyperline}
  $M^n$ contains a straight line of $\M^{n+1}$;
  \item \label{it:hypertype}
  for no $P \in M^n$, we have $\jt(P) = 1$.
\end{enumerate}
Then $M^n$ is a cylinder.
\end{theorem}

An extra condition (like our condition \eqref{it:hypertype}) as compared to the Euclidean case is most likely unavoidable for the claim to hold true -- in Section~\ref{s:counter} we construct an example of a surface $M^2 \in \M^3$ whose flag curvature is nonnegative, but which is locally saddle (that is, locally $\jt=1$). Note that in higher codimension, it gets even ``worse": by \cite[Theorem~1.4]{BI}, any two-dimensional Finsler metric admits a locally saddle embedding in some $\M^4$. In contrast, in the Riemannian case, the fact that the Ricci curvature is nonnegative forces the type to be zero \cite[Lemma~3.3.1]{BorUMN}, in any dimension and codimension.

For submanifolds of codimension $2$ we prove the following.
\begin{theorem} \label{th:codim2}
Let $M^n$ be a complete, connected, smooth, regular submanifold in a Minkowski space $\M^{n+2}$. Suppose that
\begin{enumerate}[{\rm (a)}]
  \item \label{it:c2ric}
  the Ricci curvature of the induced Finsler metric on $M^n$ is nonnegative;
  \item \label{it:c2mu}
  for all $P \in M^n$, we have $\mu(P) = k$, where $k \ge 1$; 
  \item \label{it:c2type}
  for no $P \in M^n$, we have $\jt(P) \in \{1, 2\}$.
\end{enumerate}
Then $M^n$ is a cylinder with a $k$-dimensional generatrix.
\end{theorem}

A submanifold $M^n$ in a Minkowski space $\M^{n+p}$ is said to be \emph{$k$-ruled, $ k \ge 1$}, if $M^n$ is locally foliated by domains of $k$-dimensional affine subspaces of $\M^{n+p}$. We have the following local result.

\begin{theorem} \label{th:ruled}
Let $M^n$ be a smooth, regular submanifold in a Minkowski space $\M^{n+p}$. Suppose $M^n$ is $k$-ruled. If the Ricci curvature of the induced Finsler metric on $M^n$ is nonnegative, then $\mu(P) \ge k$, for all $P \in M^n$.
\end{theorem}

For surfaces, this implies the following global fact.
\begin{theorem} \label{th:2dim}
Suppose $M^2$ is a complete, connected, smooth, regular, ruled surface in a Minkowski space $\M^{2+p}$. If the flag curvature of the induced Finsler metric on $M^2$ is nonnegative, then $M^2$ is a cylinder.
\end{theorem}

Throughout the paper, ``regular submanifold" means an immersed submanifold. The differentiability conditions can be relaxed from $C^\infty$ to $C^4$ for the Minkowski norm and to $C^3$ for the submanifold.

\section{Preliminaries}
\label{s:pre}

Let $\M^{n+p}=(\br^{n+p}, F), \; n \ge 2, \, p \ge 1$, be a Minkowski space with the Minkowski function $F=F(y^1,\dots,y^{n+p})$, where $y^a, \, a=1,\dots,n+p$, are Cartesian coordinates on $\br^{n+p}$. Denote $H=\frac12 F^2$. The function $H$ is positively homogeneous of degree $2$, smooth outside of the origin, and satisfies $\Hess(H)>0$ outside of the origin.

We adopt the following index convention: $i, j, k, l, r, s = 1, \dots, n; \; \alpha, \beta = n+1, \dots , n+p$; $a,b = 1, \dots ,n + p$. In all summations below, the indices run over the corresponding ranges.


Let $M^n \subset \M^{n+p}$ be a regular submanifold and let $P \in M^n$. We choose the Cartesian coordinates $y^a$ in such a way that $P$ is the origin of $\br^{n+p}$ and that $M^n$ is locally defined by
\begin{equation*}
y^\alpha =f^\alpha(x^1,\dots, x^n), \quad y^i=x^i,
\end{equation*}
with $f^\alpha$ smooth functions satisfying $f^\alpha(0)=f^\alpha_i(0)=0$, where here and below the function name with a subscript denotes the partial derivative.

The induced Finsler metric on the submanifold $M^n$ is given by
\begin{equation*}
S(x,u)=F(u^1,\dots, u^n, u^if^1_i, \dots, u^if^p_i),
\end{equation*}
where here and below we adopt the Einstein summation convention.

The fundamental tensor on $M^n$ is given by
\begin{equation*}
    g_{ij}=\frac{\partial^2}{\partial u^i \partial u^j} \Big( \frac12 S^2 \Big) = H_{ij} + H_{i\alpha} f^\alpha_j + H_{j\alpha} f^\alpha_i + H_{\alpha\beta} f^\alpha_i f^\beta_j,
\end{equation*}
where the partial derivatives of $H$ are evaluated at the points of $M^n$ \cite{CS}.

In particular, at the origin $x=0$ we have 
\begin{equation}\label{eq:gat0gen}
g_{ij}(0,u)=H_{ij}(0,u).
\end{equation}

The spray coefficients of $M^n$ are given by
\begin{equation*}
    G^i=\frac14 g^{ij} \Big(u^k \frac{\partial^2}{\partial x^k \partial u^j} \big(S^2\big) - \frac{\partial}{\partial x^j} \big(S^2\big)\Big)= \frac12 g^{ij} f^\alpha_{kl} u^k u^l (H_{j\alpha} + H_{\alpha\beta} f^\beta_j),
\end{equation*}
and so at the origin $x=0$ we have
\begin{equation}\label{eq:Gat0gen}
G^i(0,u) = \frac12 h^i_\a f^\a_{kl} u^k u^l = \frac12 h^i_\a \kappa^\a,
\end{equation}
where we denote
\begin{equation} \label{eq:handkappa}
    h^i_\a(u)=g^{ij}(0,u) H_{j\a}(0,u) \qquad \text{and} \qquad \kappa^\a(u)=f^\a_{kl} u^k u^l.
\end{equation}

We then compute the Ricci tensor
\begin{equation*}
    R^i_k(u) = 2 \frac{\partial G^i}{\partial x^k} - u^j \frac{\partial^2 G^i}{\partial x^j \partial u^k} + 2 G^j \frac{\partial^2 G^i}{\partial u^j \partial u^k} - \frac{\partial G^i}{\partial u^j} \frac{\partial G^j}{\partial u^k}
\end{equation*}
at the origin $x=0$. A straightforward but somewhat tedious calculation gives
\begin{equation} \label{eq:Rik1}
\begin{split}
    R^i_k(u) =& - \frac12 \frac{\partial h^i_\a}{\partial u^k} f^\a_{jls} u^j u^l u^s \\
    & - \Big(\frac34 \frac{\partial h^i_\a}{\partial u^j} \frac{\partial h^j_\b}{\partial u^k} + \frac12 \frac{\partial}{\partial u^k} (g^{ij} \frac{\partial}{\partial u^j} (H_{\a\b}-h^s_\a H_{s\b}) \Big) \kappa^\a \kappa^\b \\
    & + \Big(\frac32 \frac{\partial h^i_\a}{\partial u^k} h^j_\b - \frac12 \frac{\partial}{\partial u^k} (g^{ij}(H_{\a\b} - h^s_\b H_{s\a})) \Big) \kappa^\b f^\a_{jl} u^l \\
    & -\frac12 g^{ij} (\frac{\partial}{\partial u^j} (H_{\a\b} - h^s_\a H_{s\b})) \kappa^\b f^\a_{kl} u^l \\
    & + g^{il}(H_{\a\b}-h^s_\a H_{s\b})(f^\b_{kl}f^\a_{jr}-f^\b_{jl}f^\a_{kr}) u^j u^r,
\end{split}
\end{equation}
where all the partial derivatives of $H$ are evaluated at the point $(u,0)$ and all the partial derivatives of $f$, at the point $x=0$. It follows that the Ricci curvature of the induced Finsler metric at the origin in the direction $u$ is given by
\begin{equation} \label{eq:Ric1}
    \Ric(u) =R^i_i(u)
    = \xi_{\a} f^\a_{jls} u^j u^l u^s + \zeta_{\a\b}g^{il} (f^\b_{il}f^\a_{jr}-f^\b_{jl}f^\a_{ir}) u^j u^r - \eta_{\a\b} \kappa^\a \kappa^\b - \rho_{\a\b}^i \kappa^\b f^\a_{il} u^l,
\end{equation}
where
\begin{equation} \label{eq:Riccoefs}
\begin{alignedat}{2}
    \xi_{\a}& = - \frac12 \frac{\partial h^i_\a}{\partial u^i}, & \qquad \zeta_{\a\b}& = H_{\a\b}-h^s_\a H_{s\b}, \\
    \eta_{\a\b}& = \frac34 \frac{\partial h^i_\a}{\partial u^j} \frac{\partial h^j_\b}{\partial u^i} + \frac12 \frac{\partial}{\partial u^i} \Big(g^{ij} \frac{\partial \zeta_{\a\b}}{\partial u^j}\Big), &  \qquad \rho_{\a\b}^i & = 3 \xi_\a h^i_\b + \frac12 \frac{\partial}{\partial u^j} (g^{ij}\zeta_{\a\b}) + \frac12 g^{ij} \frac{\partial \zeta_{\a\b}}{\partial u^j}.
\end{alignedat}
\end{equation}

\bigskip

Later in the proofs we will use the following simple fact.

\begin{lemma} \label{l:Habpos}
The $p \times p$ matrix $\zeta_{\a\b}$ is positive definite.
\end{lemma}
\begin{proof}
From \eqref{eq:Riccoefs} and \eqref{eq:handkappa} we have
$\zeta_{\a\b} = H_{\a\b}-h^s_\a H_{s\b} = H_{\a\b} - g^{sr} H_{r\a} H_{s\b}$. By specifying the coordinates $y^\alpha$ on $\br^{n+p}$ it suffices to prove that $H_{n+1,n+1} - g^{rs} H_{r,n+1} H_{s,n+1} > 0$. But from the cofactor decomposition, the latter expression equals $\det H[n+1]/\det H[n]$, where $H[m]$ denotes the top left $m \times m$ submatrix of $\Hess(H)$.
\end{proof}

\section{Proofs of the theorems}
\label{s:proofs}

\begin{proof}[Proof of Theorem~\ref{th:hyper}]

We start with the following local fact.
{
\begin{proposition} \label{p:hyper}
Let $M^n$ be a regular hypersurface in a Minkowski space $\M^{n+1}$ and let $P \in M^n$. Suppose that the Ricci curvature of the induced Finsler metric on $M^n$ is nonnegative at $P$ and that $\jt(P) \ne 1$. Then for any choice of a Euclidean metric on $\M^{n+1}$, the \emph{(}Riemannian\emph{)} second fundamental form of $M^n$ at $P$ is semidefinite and the \emph{(}Ricci and sectional\emph{)} curvature of the induced Riemannian metric on $M^n$ is nonnegative at $P$.
\end{proposition}
\begin{proof}
  Choose the Cartesian coordinates $y^1, \dots, y^n, y^{n+1}$ on $\M^{n+1}$ as in Section~\ref{s:pre} and introduce the Euclidean metric on $\br^{n+1}$ in such a way that they are orthonormal. From \eqref{eq:Ric1} (dropping the indices $\a$ and $\b$, as $p=1$), the Ricci curvature in the direction $u=(u^i)$ at $P$ is given by
  \begin{equation} \label{eq:Richyper}
    \Ric(u) = \xi(u) f_{jls} u^j u^l u^s + \kappa(u) \phi(u) - \zeta(u) g^{il}(u) f_{jl} f_{ir} u^j u^r,
  \end{equation}
  where $\phi(u) = -\eta(u) \kappa(u) - \rho^i(u) f_{il} u^l + \zeta(u) g^{il} f_{il}$. Note that the matrix of the second fundamental form of $M^n$ at $P$ relative to the unit normal $\nu=(0, \dots, 0, 1)$ is given by $(f_{ij})$ and that $\kappa(u)=f_{ij}u^iu^j$ is the normal curvature of $M^n$ at $P$ in the direction of $u$. Let $C$ be the isotropic cone in the orthogonal complement to the null-space $\Lc(P)$, that is, $C=\{u \, | \, \kappa(u)=0, \, u \perp \Lc(P)\}$ and let $C_*=C \setminus \{0\}$.

  Suppose that $\jt(P) > 1$. Then the set $C_*$ is nonempty and connected, and so on a curve $\gamma \subset C_*$ which joins an arbitrary $v \in C_*$ to $-v \in C_*$, there exists a point $u$ such that $f_{jls} u^j u^l u^s=0$. From \eqref{eq:Richyper} it follows that $\zeta(u) g^{il}(u) f_{jl} f_{ir} u^j u^r \le 0$. But $\zeta(u) > 0$ by Lemma~\ref{l:Habpos}. It follows that $f_{jl} u^j=0$, so that $u \in \Lc(P)$, a contradiction.

  Therefore $\jt(P) = 0$, and so the second fundamental form at the point $P$ is semidefinite. This property does not depend on the choice of the Euclidean metric on $\br^{n+1}$ and moreover, by the Gauss equation, implies that both the sectional and the Ricci curvature of $M^n$ at $P$ are nonnegative.
\end{proof}
}

Returning to the proof of the theorem, choose an arbitrary Euclidean metric on $\br^{n+1}$. As the induced Riemannian metric on the hypersurface $M^n$ has   nonnegative curvature and as $M^n$ contains a line, it is isometric to the Riemannian product $M^{n-1} \times \E^1$, with the line being $Q \times \E^1$ for some point $Q \in M^{n-1}$ \cite{Top}. The proof is now concluded by the following Lemma (note that its hypothesis imposes no restrictions on the curvature).

\begin{lemma*}[\cite{Borline}, {\cite[Lemma~3.3.2]{BorUMN}}] 
Suppose a complete Riemannian manifold $N^n$ is isometric to the metric product $N^{n-k} \times \E^k$, and let $\iota: N^n \to \E^{n+p}$ be an isometric $C^0$-immersion. If for some $Q \in N^{n-k}$ the image $\iota(Q \times \E^k)$ is a $k$-dimensional affine subspace of $\E^{n+p}$, then $\iota(N^n)$ is a cylinder with a $k$-dimensional generatrix. \null\hfill\qedsymbol
\end{lemma*}
\renewcommand{\qedsymbol}{}
\end{proof}

\begin{proof}[Proof of Theorem~\ref{th:codim2}]
We use the following local fact.
{
\begin{proposition} \label{p:codim2}
Let $M^n$ be a regular submanifold in a Minkowski space $\M^{n+2}$ and let $P \in M^n$. Suppose that the Ricci curvature of the induced Finsler metric on $M^n$ is nonnegative at $P$. Then $\jt(P) \le 2$.
\end{proposition}

The proof of Proposition~\ref{p:codim2} is based on the following algebraic (or rather topological) lemma. 
\begin{lemma} \label{l:pi1}
  Let $\phi_1, \phi_2$ be two quadratic forms and $\psi_1, \psi_2$ be two cubic forms on $\Rn$. Suppose that the positive index of inertia of any linear combination $\la_1\phi_1 + \la_2 \phi_2$ of maximal rank is at least $3$. Then the four forms $\phi_1, \phi_2, \psi_1, \psi_2$ have a common zero on $\Rn \setminus \{0\}$. 
\end{lemma}
Assuming Lemma~\ref{l:pi1} we give the proof of Proposition~\ref{p:codim2}.
\begin{proof}[Proof of Proposition~\ref{p:codim2}]
  Choose Cartesian coordinates $y^1, \dots, y^{n+2}$ on $\M^{n+2}$ as in Section~\ref{s:pre} and introduce the Euclidean metric on $\br^{n+2}$ in such a way that they are orthonormal. If $\jt(P) > 2$, then applying Lemma~\ref{l:pi1} to the restriction of the quadratic forms $f^\a_{ij} u_iu_j$ and the cubic forms $f^\a_{ijl} u_iu_ju_l, \; \a=n+1,n+2$, to $\Lc(P)^\perp$ we find a nonzero vector $u=(u^i) \in T_PM^n$ orthogonal to $\Lc(P)$ and such that $\kappa^\a(u)(=f^\a_{ij} u_iu_j)=0$ and $f^\a_{ijl} u_iu_ju^l=0$, for $\a=n+1, n+2$. But then from \eqref{eq:Ric1}, the Ricci curvature in the direction $u=(u^i)$ at $P$ is given by $\Ric(u) = - \zeta_{\a\b}g^{il} f^\b_{jl}f^\a_{ir} u^j u^r$. By assumption, $\Ric(u) \ge 0$, and so by Lemma~\ref{l:Habpos} we have $f^\a_{ir} u^r=0$ at $P$, for $\a=n+1, n+2$, and all $i=1, \dots, n$. But then $u \in \Lc(P)$, a contradiction.
\end{proof}
}
The claim of the theorem now follows from Proposition~\ref{p:codim2} and \cite[Theorem~2]{BorMS}.
\end{proof}

\begin{proof}[Proof of Lemma~\ref{l:pi1}]
  Denote $C$ the set (the cone) of common zeros of the quadratic forms $\phi_1$ and $\phi_2$ on $\Rn$, and let $C_*=C\setminus\{0\}$. We use a version of the Borsuk-Ulam Theorem. Let $\gamma':[0,1] \to C_*$ be a path joining two antipodal points and let $\gamma = \gamma' \cup (-\gamma')$ (that is, $\gamma:[0,2] \to C_*$ is defined by $\gamma(t)=\gamma'(t)$ for $t \in [0,1]$ and $\gamma(t)=-\gamma'(t-1)$ for $t \in [1,2]$). If the cubic forms $\psi_1, \psi_2$ have no common zeros on $C_*$, we can define the map $\Psi: C_* \to S^1$ by $\Psi(u)=(\psi_1(u), \psi_2(u))/\sqrt{\psi_1^2(u) + \psi_2^2(u)}$. Note that $\Psi$ maps antipodal points to antipodal points, and so the loop $\Psi(\gamma)$ represents a nontrivial (odd) element in $\mathbb{Z}=\pi_1(S^1)$. If we can construct the path $\gamma'$ in such a way that $\gamma$ represents an element of a finite order in $\pi_1(C_*)$, we get a contradiction (in fact, in the proof below, our $\gamma$ will always be contractible on $C_*$, with a single exception).

  The topology of $C_*$ (more precisely, of its intersection with the unit sphere $S^{n-1}$) is completely described when the pencil $\la_1\phi_1 + \la_2 \phi_2$ is \emph{generic} \cite{GL}. There are three conditions which define a generic pencil. Relative to some Cartesian coordinates on $\Rn$, the forms $\phi_1, \phi_2$ are represented by symmetric matrices $A_1, A_2$ respectively. The first condition is that $\det(\la_1 A_1 + \la_2 A_2)$ is not identically zero (and so we can assume that the matrix $A_2$ is nonsingular). The second condition is that the matrix $A_2^{-1}A_1$ is semisimple (is diagonalisable over $\bc$). The third condition is that $C \cap S^{n-1}$ is a smooth manifold; this is equivalent to the fact that for no $u \in C_*$, the vectors $A_1u$ and $A_2u$ are linear dependent. It is easy to see that by a small perturbation, any pencil can be made generic; however, in our proof we have to be a little more careful not to violate the inertia assumption of the lemma. We call the \emph{type} of a pencil of quadratic forms (or of the pencil of associated symmetric matrices relative to some basis) the minimum of the positive index of inertia of all its elements of the maximal rank. We will first show that any pencil of type at least $3$ can be arbitrarily close approximated by a generic pencil of type at least $3$. Then using the result of \cite{GL} and the arguments in the first paragraph we will show that for every such generic pencil $\la_1\phi'_1 + \la_2 \phi'_2$, the forms $\phi'_1, \phi'_2, \psi_1$ and $\psi_2$ have a common zero on the unit sphere $S^{n-1} \subset \Rn$. Then the claim of the lemma follows by compactness.

  Without loss of generality we can assume that the common null-space of $\phi_1$ and $\phi_2$ is trivial (as otherwise we can restrict all four forms $\phi_1, \phi_2, \psi_1, \psi_2$ to its orthogonal complement). Furthermore, using the canonical form of a pencil of symmetric matrices given in \cite[Theorem~9.2]{LR} we obtain that the condition $\det(\la_1 A_1 + \la_2 A_2) \equiv 0$ implies that $\Rn$ splits into the direct sum of subspaces $V$ and $V'$ orthogonal with respect to both $\phi_1$ and $\phi_2$, with $\dim(V)=2m+1, \; m \ge 1$, and such that the restriction of $\la_1A_1+\la_2A_2$ to $V$ relative to a particular basis for $V$ is given by $Q(\la_1,\la_2)=\left(\begin{smallmatrix} 0 & T(\la_1,\la_2) \\ T^t(\la_1,\la_2) & 0 \end{smallmatrix}\right)$, where $T(\la_1,\la_2)$ is the $m \times (m+1)$ matrix of the form 
  \begin{equation*}
    T(\la_1,\la_2)=\left(
                     \begin{array}{cccc}
                       0 &  & \la_1 & \la_2 \\
                        & \iddots & \iddots &  \\
                       \la_1 & \la_2 &  & 0 \\
                     \end{array}
                   \right).
  \end{equation*}
  Note that the pencil $Q(\la_1,\la_2)$ has an isotropic subspace $V_0$ of dimension $m+1$, which is a common isotropic subspace of the forms $\phi_1$ and $\phi_2$. If $m \ge 2$, then $\dim V_0 \ge 3$. Suppose $m=1$. As the positive index of inertia of all nonzero elements of the pencil $Q(\la_1,\la_2)$ is $1$, the restriction of the pencil $\la_1\phi_1 + \la_2\phi_2$ to $V'$ has type at least $2$. Then $\dim V' \ge 4$, and so by \cite[Theorem~11.5(ii)]{LR}, the forms $\phi_1$ and $\phi_2$ have a common zero $u \in V' \setminus \{0\}$, and hence a common $3$-dimensional isotropic subspace $V_0 \oplus \br u$. So in all the cases, $\phi_1$ and $\phi_2$ have a common isotropic subspace of dimension $3$. The unit sphere $S^2$ in that subspace is a subset of $C_*$. The claim of the lemma now follows from the argument in the first paragraph of the proof if we take for $\gamma$ the equator of $S^2$.

  We can therefore assume that for our pencil, $\det(\la_1 A_1 + \la_2 A_2)$ does not vanish identically, and so without loss of generality we can assume that $\det A_2 \ne 0$. We can then consider the real Jordan form of the matrix $A_2^{-1}A_1$. The Jordan cells determine the decomposition of $\Rn$ into the direct sum of subspaces $V_j$ which are pairwise orthogonal with respect to both $\phi_1$ and $\phi_2$, and such that for each subspace $V_j$, there is a basis relative to which the restrictions of the forms $\phi_1$ and $\phi_2$ are represented by a canonical pair of matrices, as given in \cite[Theorem~9.2]{LR}. These canonical pairs come in two forms, depending on whether the Jordan cell corresponds to a real eigenvalue or to a pair of complex-conjugate eigenvalues of the matrix $A_2^{-1}A_1$.

  For a subspace $V_j$ of dimension $m \ge 1$ corresponding to a real eigenvalue $\a$, the canonical form is given by the pencil of $m \times m$ symmetric matrices
  \begin{equation*}
    \la_1 Q_1 + \la_2 Q_2 =\eta ( \la_1 \left(
                     \begin{array}{cccc}
                       0 &  & 1 & \a\\
                        & \iddots & \iddots &  \\
                       1 & \iddots &  &  \\
                       \a& &  & 0 \\
                     \end{array}
                   \right) + \la_2
                   \left(
                     \begin{array}{cccc}
                       0 &  & 0 & 1\\
                        & \iddots & \iddots &  \\
                       0 & \iddots &  &  \\
                       1 & &  & 0 \\
                     \end{array}
                   \right) ),
  \end{equation*}
  where $\eta = \pm 1$. Now if $m \ge 3$, we replace the top-right and the bottom-left $\a$ in $Q_1$ by $\a+\ve$, with a small $\ve$. It is easy to see that the inertia of the resulting pencil does not change (outside a finite set of points on the circle $\la_1^2 + \la_2^2=1$), but the matrix $Q_2^{-1}Q_1$ will now have two different eigenvalues, $\a$ and $\a+\ve$, so that the pencil $\la_1 Q_1 + \la_2 Q_2$ splits into the direct sum of two smaller pencils of the same form. If $m=2$, we replace the bottom-right zero in $Q_1$ by $-\ve$, with a small $\ve>0$. Again, the inertia of the resulting pencil is the same (outside a finite set of points on the circle $\la_1^2 + \la_2^2=1$), but the resulting matrix $Q_2^{-1}Q_1$ has two non-real complex-conjugate eigenvalues, $\a \pm \mathrm{i}\sqrt{\ve}$. So by an arbitrarily small perturbation without changing the type we can obtain a pencil such that all the subspaces $V_j$ corresponding to the real eigenvalues of the matrix $A_2^{-1}A_1$ are $1$-dimensional. The restrictions of the (perturbed) forms $\phi_1, \phi_2$ to the sum of these subspaces are simultaneously diagonalisable over $\br$.

  For a subspace $V_j$ of dimension $2m, \; m \ge 1$, corresponding to the pair of complex-conjugate eigenvalues $\rho \pm \mathrm{i} \nu$ of $A_2^{-1}A_1$, the canonical form is given by the pencil of $2m \times 2m$ symmetric matrices
  \begin{equation*}
    \la_1 Q_1 + \la_2 Q_2 = \la_1 \left(
                     \begin{array}{cccc}
                       0 &  & K & L\\
                        & \iddots & \iddots &  \\
                       K & \iddots &  &  \\
                       L & &  & 0 \\
                     \end{array}
                   \right) + \la_2
                   \left(
                     \begin{array}{cccc}
                       0 &  & 0 & K\\
                        & \iddots & \iddots &  \\
                       0 & \iddots &  &  \\
                       K & &  & 0 \\
                     \end{array}
                   \right),
    \quad
    \begin{aligned}
    K&=\left(
        \begin{array}{cc}
          0 & 1 \\
          1 & 0 \\
        \end{array}
      \right), \\
    L&=\left(
        \begin{array}{cc}
          \nu & \rho \\
          \rho & -\nu \\
        \end{array}
      \right).
  \end{aligned}
  \end{equation*}
  If $m \ge 3$, we replace $\rho$ by $\rho+\ve$ in the top-right and the bottom-left blocks $L$ in $Q_1$, with a small $\ve$. This does not change the inertia of the pencil, but the matrix $Q_2^{-1}Q_1$ for the resulting pencil will have two different pairs of complex-conjugate eigenvalues, $\rho \pm \mathrm{i} \nu$ and $\rho+\ve \pm \mathrm{i} \nu$, and so the pencil $\la_1 Q_1 + \la_2 Q_2$ splits into the direct sum of two smaller pencils of the same form. If $m=2$, we replace the bottom-right $2 \times 2$ zero block in $Q_1$ by $\ve I_2$, with a small $\ve$. It is not hard to see that the inertia of the resulting pencil is the same, but the resulting matrix $Q_2^{-1}Q_1$ has two different pairs of complex-conjugate eigenvalues, $\rho \pm \mathrm{i}\sqrt{\nu^2 + \ve}$ and $\rho \pm \mathrm{i}\sqrt{\nu^2 - \ve}$. Then the pencil $\la_1 Q_1 + \la_2 Q_2$ splits into the direct sum of two pencils of the same form, with $m=1$.

  So by an arbitrarily small perturbation of the given pencil, we can obtain a pencil of the same type such that the matrix $A_2^{-1}A_1$ is semisimple (over $\bc$), and we can chose coordinates $(u_i, v_t, w_t)$ on $\Rn$ such that the quadratic forms $\phi_1$ and $\phi_2$ are given by
  \begin{equation} \label{eq:subcan}
    \phi_1 = \sum_{i=1}^r \a_i u_i^2 + \sum_{t=1}^s (\nu_t (v_t^2 -w_t^2) + 2\rho_t v_tw_t), \quad \phi_2 = \sum_{i=1}^r \beta_i u_i^2 + \sum_{t=1}^s 2v_tw_t,
  \end{equation}
  where $r, s \ge 0, \, r+2s=n, \; \beta_i \ne 0$ and $\nu_t > 0$.

  Next, we want $C \cap S^{n-1}$ to be a smooth manifold. This is equivalent to the fact that for no $u \in C_*$ the vectors $A_1u$ and $A_2u$ are linear dependent, which is equivalent to the fact that the origin of $\br^2$ does not belong to the convex hull of any two vectors $(\a_i, \b_i)$ and $(\a_j, \b_j), \; 1 \le i, j \le r$. Assuming this condition is violated, we can replace $\phi_1$ by $\phi_1 + 2 \ve u_iu_j$ for some small $\ve$, without changing $\phi_2$; it is easy to check that the type of the pencil does not change (but note that $r$ decreases by $2$ and $s$ increases by $1$). Repeating this procedure if necessary we can assume that the quadratic forms $\phi_1$ and $\phi_2$ are given by \eqref{eq:subcan}, and for no two $i, j$, the origin of $\br^2$  belongs to the convex hull of the vectors $(\a_i, \b_i)$ and $(\a_j, \b_j)$.

  We will now continuously deform the coefficients $\a_i, \b_i, \nu_t, \rho_t$ in \eqref{eq:subcan}, without violating the above condition on the pairs of vectors $(\a_i, \b_i)$ and $(\a_j, \b_j)$ and keeping $\nu_t$ positive. Note that the inertia of the individual elements of the pencil may change in the process of deformation (for example, we may create or destroy multiple eigenvalues of $A^{-1}_2A_1$), but it is not hard to see that the type of the pencil remains the same. Moreover, the diffeomorphism type of the manifold $C \cap S^{n-1}$ does not change, and we can reduce the forms $\phi_1$ and $\phi_2$ to the following canonical forms \cite[Theorem~1]{AG}:
  \begin{equation} \label{eq:can}
    \phi_1 = \sum_{j=1}^{2l-1} \Bigl(\cos \tfrac{2\pi j}{2l-1} \sum_{i_j=1}^{n_j} u_{i_j}^2\Bigr) + \sum_{t=1}^s (v_t^2 -w_t^2), \quad
    \phi_2 = \sum_{j=1}^{2l-1} \Bigl(\sin \tfrac{2\pi j}{2l-1} \sum_{i_j=1}^{n_j} u_{i_j}^2\Bigr) + \sum_{t=1}^s 2v_tw_t,
  \end{equation}
  where $r = \sum_{j=1}^{2l-1} n_j \ge 0, \, s \ge 0, \, r+2s=n$, and either $l=0$ (we then take $r=0$), or $l > 0$ and then $n_j \ge 1$, for all $j = 1, \dots, 2l-1$. If $l \ge 2$, we introduce the numbers $d_j=n_j + \dots + n_{j+l-2}, \; j=1, \dots, 2l-1$, where the indices are computed modulo $2l-1$. Then the type of the pencil defined by $\phi_1$ and $\phi_2$ equals $s+\min(d_j)$ when $l \ge 2$, and equals $s$ when $l \le 1$. Therefore by our assumption, we have the following inequalities:
  \begin{equation}\label{eq:typebi}
    s \ge 3, \quad \text{when } l \le 1; \qquad \qquad s+d_j \ge 3, \; j=1, \dots, 2l+1, \quad \text{when } l \ge 2.
  \end{equation}

  By \cite{GL} the submanifold $C \cap S^{n-1}$ is diffeomorphic to
  \begin{enumerate}[(i)]
    \item
    the unit tangent bundle of $S^{s-1}$, if $r = 0, \, s > 1$;

    \item
    the product $S^{s-1} \times S^{r+s-2}$, if $r > 0, \, l=1, \, s > 0$;

    \item
    the product $S^{n_1-1} \times S^{n_2-1} \times S^{n_3-1}$, if $r > 0, \, l = 2, \, s = 0$;

    \item
    the connected sum $\#_{j=1}^{2l-1} (S^{d_j+s-1} \times S^{r-d_j+s-2})$, if $r > 0, \, l \ge 2, \, l + s > 2$
  \end{enumerate}
  (there is also a case $C \cap S^{n-1} = \varnothing$ when $r=0, \; s \le 1$, or when $r > 0, \; l=1, \; s = 0$, which may not occur with our assumption on the type). But then from inequalities \eqref{eq:typebi} we can easily see that all the spheres in the above classification are of dimension at least $2$, and so $C \cap S^{n-1}$ is always connected and is simply connected in the last three cases, and $\pi_1(C \cap S^{n-1})=\mathbb{Z}_2$ in the first case. This concludes the proof of the lemma.
\end{proof}

Note that for the proof, we need much less than the diffeomorphism type of $C \cap S^{n-1}$: it is sufficient to show that the condition on the type of the pencil implies that $C_*$ is connected and that its first Betti number is zero. Instead of the classification theorem of \cite{GL} one could have used the results of \cite[\S2]{Kr}; however, we cannot see how to avoid the linear-algebraic ``preparation".

In the Riemannian case, the type of a point of a submanifold of an arbitrary codimension in a Euclidean space at which the Ricci curvature is nonnegative is $0$ \cite[Lemma~3.3.1]{BorUMN}. This is no longer true in the Minkowski settings: the possibility $\jt(P)=1$ realises already for a two-dimensional surface of positive flag curvature in $\M^3$ (as in the example in Section~\ref{s:counter}). On the other hand, by Proposition~\ref{p:hyper} and Proposition~\ref{p:codim2}, for $p=1, 2$, the type of a point of a submanifold of codimension $p$ of a Minkowski space at which the Ricci curvature is nonnegative is at most $p$. It may be interesting to ask what happens in higher codimension. In that case, constructing the canonical form of a pencil by changing a basis is hopeless, but one may use the spectral sequence from \cite[Theorem~1]{AV} to compute the homology of the common set of zeros of the quadratic forms.

\begin{proof}[Proof of Theorem~\ref{th:ruled}]
  Let $P \in M^n$. Choose Cartesian coordinates $y^1, \dots, y^{n+p}$ on $\M^{n+p}$ as in Section~\ref{s:pre} and introduce the Euclidean metric on $\br^{n+p}$ in such a way that they are orthonormal. In a direction $u=(u^i)$ tangent to a $k$-dimensional generatrix of $M^n$ passing through $P$ (there may be more than one such generatrix) we have $\kappa^\a(=f^\a_{ij}u^iu^j)=0$ and $f^\a_{ijl}u^iu^ju^l=0$ at $P$, for all $\a=n+1, \dots, n+p$. Then from \eqref{eq:Ric1}, the Ricci curvature in the direction $u=(u^i)$ at $P$ is given by $\Ric(u) = -\zeta_{\a\b} g^{il} f^\b_{jl}f^\a_{ir} u^j u^r$. By assumption, $\Ric(u) \ge 0$, and so by Lemma~\ref{l:Habpos} we have $f^\a_{ir} u^r=0$ at $P$, for all $\a=n+1, \dots, n+p$ and all $i=1, \dots, n$. It follows that the tangent space to the $k$-dimensional generatrix passing through $P$ lies in the null-space $\Lc(P)$.
\end{proof}

\begin{proof}[Proof of Theorem~\ref{th:2dim}]
  By Theorem~\ref{th:ruled} we have $\mu(P) \ge 1$, for all $P \in M^2$. Choose an arbitrary Euclidean metric on $\br^{2+p}$. Then the induced Riemannian metric on $M^2$ is flat and complete, and so the claim follows from~\cite[Theorem~1.1]{Har}.
\end{proof}

\section{Example}
\label{s:counter}

We show that there exists a smooth Minkowski metric on $\br^3$, a smooth surface $M^2 \subset \br^3$ and a point $P \in M^2$ such that:

\begin{itemize}
  \item
  the Gauss curvature of $M^2$ at $P$ is negative (relative to any choice of a Euclidean metric on $\br^3$);
  \item
  the flag curvature of $M^2$ at $P$, with any flagpole, is nonnegative (even positive).
\end{itemize}

Define the Minkowski function $F=F(y^1,y^2,y^3)$, where $y^i, \, i=1,2,3$, are Cartesian coordinates on $\br^3$ in such a way that $H=\frac12 F^2$ is given by
\begin{equation} \label{eq:H}
\begin{split}
    H(y^1,y^2,y^3) &=A((y^1)^2+(y^2)^2) + \ve_1 y^3 \, \frac{3y^2(y^1)^2-(y^2)^3}{(y^1)^2+(y^2)^2} \\
      &+(y^3)^2 \Big(B + \ve_2 \frac {(y^1)^6 - 15 (y^1)^4 (y^2)^2 + 15 (y^1)^2 (y^2)^4 - (y^2)^6}{((y^1)^2+(y^2)^2)^3}\Big),
\end{split}
\end{equation}
where $A$ and $B$ are large positive constants, and $\ve_1$ and $\ve_2$ are small positive constants; relative to the cylindrical coordinates $r, \theta, z$, where $y^1=r \cos \theta, \, y^2=r \sin \theta, \, y^3=z$, the function $H$ is given by
\begin{equation*}
    H(r, \theta, z) = Ar^2 + \ve_1 z r \sin(3\theta)+ z^2 (B+ \ve_2 \cos(6\theta)).
\end{equation*}
Note that the Minkowski metric so defined is reversible.

Define the surface $M^2$ by $y^3=f(x^1,x^2), \; y^1=x^1, \; y^2=x^2$, with $f$ a smooth function satisfying $f(0,0)=f_i(0,0)=0$ and take $P$ to be the origin $x^1=x^2=0$.

From \eqref{eq:Ric1} and \eqref{eq:Riccoefs}, the flag curvature of the surface $M^2$ at the point $P$ with the flagpole $u=(u^1,u^2)$ is given by
\begin{equation}\label{eq:ricco}
\begin{split}
\Ric(u) & = \frac {2 \ve_1 u^2 (3 (u^1)^2 -(u^2)^2)}{A((u^1)^2+(u^2)^2)^2} f_{ijk} u^i u^j u^k
+ \frac{P(u)}{4 A^2 ((u^1)^2+(u^2)^2)^2} \, (f_{11} f_{22}-f_{12}^2) \\
& + \frac{Q_{11}(u)f_{11}+Q_{12}(u)f_{12}+Q_{22}(u)f_{22}}{A^2 ((u^1)^2+(u^2)^2)^4} \, (f_{11}(u^1)^2 + 2 f_{12} u^1 u^2 + (u)f_{22}(u^2)^2),
\end{split}
\end{equation}
where
\begin{align*}
P(u) &=(4 A B + 4 A \ve_2 - 9 \ve_1^2) (u^1)^6 + 3 (4 A B - 20 A \ve_2 + 15 \ve_1^2) (u^1)^4 (u^2)^2 \\
& \quad + 3 (4 A B + 20 A \ve_2 - 25 \ve_1^ 2) (u^1)^2 (u^2)^4 + (4 A B - 4 A \ve_2 - \ve_1^2) (u^2)^6,\\
Q_{11}(u) &= 6 (u^1)^2 ((3 A \ve_2-3 \ve_1^2) (u^1)^6+(33 \ve_1^2-51 A \ve_2) (u^1)^4(u^2)^2 \\
& \quad +(65 A \ve_2-65 \ve_1^2) (u^1)^2 (u^2)^4+(11 \ve_1^2-9 A \ve_2) (u^2)^6),\\
Q_{22}(u) &= 3 (u^2)^2 ((18 A \ve_2-27 \ve_1^2) (u^1)^6+(115 \ve_1^2-130 A \ve_2) (u^1)^4 (u^2)^2 \\
& \quad +(102 A \ve_2-81 \ve_1^2) (u^1)^2 (u^2)^4+(\ve_1^2-6 A \ve_2) (u^2)^6),\\
Q_{12}(u) &= 3 u^1 u^2 ((24 A \ve_2-33 \ve_1^2) (u^1)^6+(181 \ve_1^2-232 A \ve_2) (u^1)^4 (u^2)^2\\
& \quad +(232 A \ve_2-211 \ve_1^2) (u^1)^2 (u^2)^4+(23 \ve_1^2-24 A \ve_2) (u^2)^6).
\end{align*}

We now choose the third derivatives $f_{ijk}$ at zero in such a way that
\begin{equation*}
f_{111} (u^1)^3 + 3 f_{112} (u^1)^2 u_2 + 3 f_{122} u^1 (u^2)^2 + f_{222} (u^2)^3= C u^2 (3 (u^1)^2-(u^2)^2),
\end{equation*}
where $C$ is a very large positive number. Then the first term in \eqref{eq:ricco} is large and positive everywhere except for on the three lines $\br u, \; u=(1,0), (1, \pm \sqrt{3})$, where it is zero. It therefore suffices to choose the second derivatives $f_{ij}$ at zero in such a way that the sum $T(u)$ of the other two terms in \eqref{eq:ricco} is positive at these three points and that $f_{11} f_{22}-f_{12}^2 < 0$. We have
\begin{align*}
    T(1,0) & = \frac{1}{4 A^2}(a_1(f_{11} f_{22}-f_{12}^2) + 72 a_2 f_{11}^2), \\
    T(1,\pm\sqrt{3})&= \frac{1}{2 A^2} (2 a_1 (f_{11} f_{22} - f_{12}^2) + 9 a_2 (f_{11} + 3 f_{22} - 2 \sqrt{3} f_{12})^2),
\end{align*}
where $a_1=4 A B + 4 A \ve_2-9 \ve_1^2, \; a_2= A \ve_2 - \ve_1^2$. Take $f_{11}=f_{22}=1, \, f_{12}=\sqrt{1+\ve_3}$, where $\ve_3 > 0$. Then $f_{11} f_{22} - f_{12}^2 < 0$, and for $\ve_1$ and $\ve_3$ small enough, $T(1,0), T(1, \pm \sqrt{3}) > 0$.

\bibliographystyle{plain}
\bibliography{cylinderarx}

\end{document}